\documentclass[11pt,a4paper]{amsart}
\usepackage{amssymb, amstext, amscd, amsmath}
\usepackage{tikz}
\usepackage{url}
\usepackage{amsfonts}
\usepackage{lscape}
\usepackage{amsthm}
\usepackage{amsfonts}
\usepackage{array}
\usepackage{eucal}
\usepackage{latexsym}
\usepackage{mathrsfs}
\usepackage{textcomp}
\usepackage{verbatim}

\newtheorem{thm}{Theorem}[section]

\newtheorem{con}[thm]{Conjecture}

\newtheorem{lem}[thm]{Lemma}
\newtheorem{prop}[thm]{Proposition}

\numberwithin{equation}{section}

\newcommand{\bQ}{{\mathbb{Q}}}
\newcommand{\bR}{{\mathbb{R}}}

\newcommand{\bK}{{\mathbb{K}}}
\newcommand{\bL}{{\mathbb{L}}}
  \newcommand{\C}{{\mathcal{C}}}
  
  \newcommand{\E}{{\mathcal{E}}}
  \newcommand{\F}{{S}}

  \newcommand{\M}{{\mathcal{M}}}

\renewcommand{\S}{{\mathcal{S}}}

  \newcommand{\W}{{\mathcal{W}}}
  
  \newcommand{\Y}{{\mathcal{Y}}}

\newcommand{\rank}{\operatorname{rank}}
\newcommand{\coker}{\operatorname{coker}}

\newcommand{\td}{\mathrm{td}\,}

\begin{document}

\title[Stress matrices and global rigidity of frameworks on surfaces]{Stress matrices and global rigidity of frameworks on surfaces}
\author[B. Jackson]{B. Jackson}
\address{School of Mathematical Sciences\\ Queen Mary, University of London\\
E1 4NS \\ U.K. }
\email{b.jackson@qmul.ac.uk}
\author[A. Nixon]{A. Nixon}
\address{Department of Mathematics and Statistics\\ Lancaster University\\
LA1 4YF \\ U.K. }
\email{a.nixon@lancaster.ac.uk}
\date{\today}

\begin{abstract}
In 2005, Bob Connelly showed that a generic framework in $\bR^d$ is
globally rigid if it has a stress matrix of maximum possible rank,
and that this sufficient condition for generic global rigidity is
preserved by the 1-extension operation. His results gave a key step
in the characterisation of generic global rigidity in the plane. We
extend these results to frameworks on surfaces in $\bR^3$. For a
framework on a family of concentric cylinders, cones or ellipsoids,
we show that there is a natural surface stress matrix arising from
assigning edge and vertex weights to the framework, in equilibrium
at each vertex.  In the case of cylinders and ellipsoids,
we show that  having a maximum rank stress matrix is
sufficient to guarantee generic global rigidity on the surface. We
then show that this sufficient condition for generic global rigidity is preserved under 1-extension and use this to
make progress on the problem of characterising generic global
rigidity on the cylinder. 
\end{abstract}

\maketitle

\section{Introduction}

A bar-joint framework $(G,p)$ in $d$-dimensional Euclidean space $\bR^d$ is a realisation of a (simple) graph $G=(V,E)$, via a map $p:V\to \bR^d$, with vertices considered as universal joints and edges as fixed length bars. A framework $(G,p)$ is \emph{rigid} if the only continuous motions of the vertices in $\bR^d$ that preserve
the edge lengths, arise from isometries of $\bR^d$.
More strongly, $(G,p)$ is \emph{globally rigid} in $\bR^d$ if every realisation $(G,q)$ with the same edge lengths as $(G,p)$ arises from an isometry of $\bR^d$. We refer the reader to  \cite{Wlong} for more information on rigidity and its applications.

It is NP-hard to determine if an arbitrary framework is rigid \cite{Abb} or globally rigid \cite{Sax}. These problems become more tractable if we restrict to generic frameworks, for which rigidity and global rigidity can be determined in polynomial time when $d=1,2$.
It remains a difficult open problem to characterise, in an efficient combinatorial way, when a $3$-dimensional generic
framework is rigid or globally rigid.

Results have recently been obtained concerning the rigidity and global rigidity of frameworks in $\bR^3$ that are constrained to lie on a fixed surface \cite{JMN,NOP,NOP2}. In this paper we obtain a natural sufficient condition for such a framework to be globally rigid.

We first recall
some fundamental results about the generic (global) rigidity of bar-joint frameworks in Euclidean space.
A graph $G=(V,E)$ is \emph{$(2,k)$-sparse} if for every subgraph
$G'=(V',E')$, with at least one edge, the inequality $|E'|\leq
2|V'|-k$ holds. Moreover $G$ is \emph{$(2,k)$-tight} if $|E|=2|V|-k$
and $G$ is $(2,k)$-sparse. While the definitions of {$(2,k)$-sparse}
and {$(2,k)$-tight} make sense for graphs with multiple edges and
loops, such edges are of no use when considering rigidity so we will
restrict our attention to simple graphs.

\begin{thm}[\cite{Lam}]\label{thm:Laman}
A generic framework $(G,p)$ in $\bR^2$ is
rigid if and only if $G$ contains a spanning $(2,3)$-tight subgraph.
\end{thm}

A framework $(G,p)$ is said to be {\em redundantly rigid} if
$(G-e,p)$ is rigid for all edges $e$ of $G$.

\begin{thm}[\cite{C2005,J&J}]\label{thm:2dcomb}
A generic framework  $(G,p)$ in $\bR^2$ is globally rigid if and only if $G$ is a complete graph on at most three vertices or $(G,p)$ is redundantly rigid and $G$ is $3$-connected.
\end{thm}

Hendrickson \cite{Hdk} had previously shown that $(d+1)$-connectivity and redundant rigidity are necessary conditions for generic global rigidity in $\bR^d$
for all $d$. Examples constructed by Connelly \cite{C91} show that they are not sufficient to imply generic global rigidity when $d\geq 3$.
Connelly also obtained a different kind of sufficient condition for generic global rigidity in terms of `stress matrices' (which will be defined in Section \ref{sec:stress}).

\begin{thm}[\cite{C2005}]\label{thm:Consufficiency}
Let $(G,p)$ be a generic framework in $\bR^d$ with $n\geq d+2$ vertices. Suppose that $(G,p)$ has an equilibrium stress $\omega$ whose associated stress matrix $\Omega$ has
rank $n-d-1$. Then $(G,p)$ is globally rigid in $\bR^d$.
\end{thm}

Gortler, Healy and Thurston \cite{GHT} have shown that Connelly's sufficient condition for generic global rigidity is also a necessary condition. This characterization implies that generic global rigidity depends only on the underlying graph (but does not give rise to a polynomial algorithm for deciding which graphs are generically globally rigid in $\bR^d$).

In this paper we develop natural analogues of an equilibrium  stress and a stress
matrix for frameworks constrained to lie on a surface. Our main result is an analogue of Theorem \ref{thm:Consufficiency}, giving a sufficient condition for generic frameworks on families of concentric cylinders 
and ellipsoids to be globally rigid.

We conclude the introduction by giving a short outline of what follows.
Section \ref{sec:surf} recalls basic definitions and results for frameworks on surfaces.
We describe the rigidity map and rigidity matrix for surfaces in Section \ref{sec:gen}. In Section \ref{sec:stress} we develop basic properties of stresses, stress matrices and energy functions for frameworks on surfaces.
Section \ref{sec:main} contains our main result, Theorem \ref{thm:global}, an analogue of Theorem \ref{thm:Consufficiency} for generic frameworks on cylinders and ellipsoids. We use this result in Section \ref{sec:1ext} to show that the property of having a maximum rank surface stress matrix is preserved by 1-extensions on these surfaces. We finish by applying our results to make some progress on the problem of characterising generic global rigidity on the cylinder.

\section{Frameworks on Surfaces}\label{sec:surf}

It was shown in \cite{NOP} that the rigidity of a generic framework on a surface depends crucially  on the number of continuous isometries of $\bR^3$ admitted by the surface, see Theorems \ref{thm:cylinderlaman}, \ref{thm:conelaman} and \ref{thm:surfacerank} below. Since generic rigidity and global rigidity on the plane and sphere \cite{C&W}, the surfaces with 3-dimensional isometry groups, are now well understood,  we consider cylinders, cones and ellipsoids as natural examples of surfaces with $2$, $1$ and $0$-dimensional isometry groups, respectively.

Let $r=(r_1,r_2,\dots, r_n)$ be a vector of (not necessarily distinct) positive real numbers. For $1\leq i \leq n$, let
$\Y_i=\{(x,y,z) \in \bR^3:x^2+y^2=r_i\}$, $\C_i=\{(x,y,z) \in \bR^3:x^2+y^2=r_iz^2\}$ and $\E_i=\{(x,y,z) \in \bR^3:x^2+\alpha y^2+\beta z^2=r_i\}$ for some fixed $\alpha,\beta \in \bQ$ with $1 < \alpha < \beta$. Let
$\Y=\bigcup_{i=1}^n\Y_i$, $\C=\bigcup_{i=1}^n\C_i$ and $\E=\bigcup_{i=1}^n\E_i$.


We will use $\F=\bigcup_{i=1}^n \F_i$ to denote one of the three surfaces $\Y,\C,\E$,  and $\ell$ for the dimension of its space of infinitesimal isometries (so $\ell=2,1$ or $0$ when $\F=\Y,\C$ or $\E$, respectively). We will occasionally use $\F(r)$ when we wish to specify a particular vector $r$ and $\F(1)$ for the special case when $r_1=r_2=\dots = r_n$ (there is no loss in generality  in assuming that $r_i=1$ for all $1\leq i\leq n$ when the $r_i$ are all equal).

A {\em framework on $\F$} is a pair $(G,p)$ where $G=(V,E)$ is a
graph with $V=\{v_1,v_2,\ldots,v_n\}$, and $p:V\to \bR^3$ with
$p(v_i)\in \F_i$ for all $1\leq i\leq n$. Two frameworks $(G,p)$ and
$(G,q)$ on $\F$ are \emph{equivalent}  if
$\|p(v_i)-p(v_j)\|=\|q(v_i)-q(v_j)\|$ for all $v_iv_j \in E$ and
\emph{congruent}  if $\|p(v_i)-p(v_j)\|=\|q(v_i)-q(v_j)\|$ for all
$v_i,v_j \in V$. The framework $(G,p)$ is \emph{globally rigid on $\F$} if every
framework $(G,q)$ on $\F$ which is equivalent to $(G,p)$ is
congruent to $(G,p)$.
We say that $(G,p)$ is \emph{rigid on $\F$} if
there exists an $\epsilon>0$ such that every framework $(G,q)$ on
$\F$ which is equivalent to $(G,p)$, and has
$\|p(v_i)-q(v_i)\|<\epsilon$ for all $1\leq i\leq n$, is congruent
to $(G,p)$. (This is equivalent to saying that every continuous
motion of the vertices that stays on $\F$ and preserves equivalence
also preserves congruence). If $(G,p)$ is not rigid on $\F$ then
$(G,p)$ is said to be \emph{flexible on $\F$}. The framework $(G,p)$
is \emph{minimally rigid on $\F$} if it is rigid and, for every edge $e\in
E$,  $(G-e,p)$ is flexible on $\F$. It is \emph{redundantly rigid on
$\F$} if $(G-e,p)$ is rigid on $\F$  for all $e\in E$.

An \textit{infinitesimal flex} $s$ of $(G, p)$ on $\F$ is a map
$s:V\to \bR^3$ such that $s(v)$ is tangential to $\F$ at $p(v)$ for
all $v\in V$ and $(p(u)-p(v))\cdot(s(u)-s(v))=0$ for all $uv\in E$.
The framework $(G,p)$ is \emph{infinitesimally rigid} on $\F$ if
every infinitesimal flex $s$ of $(G,p)$ satisfies
$(p(u)-p(v))\cdot(s(u)-s(v))=0$ for all $u,v\in V$.

We consider a framework $(G,p)$ on $\F=\F(r)$ to be \emph{generic} if $\td[\bQ(r,p):\bQ(r)]=2|V|$, where $\td[\bQ(r,p):\bQ(r)]$ denotes the transcendence degree of the field extension. Thus  $(G,p)$ is generic on $\F$ if the coordinates of the vertices of $G$ are as algebraically independent as possible.
The following results characterise when a generic framework on $\Y$
or $\C(1)$ is minimally rigid.

\begin{thm}[\cite{NOP}]\label{thm:cylinderlaman}
Let $(G,p)$ be a generic framework on $\Y$. Then $(G,p)$ is
minimally rigid if and only if $G$ is a complete graph on at most
three vertices or $G$ is $(2,2)$-tight.
\end{thm}

\begin{thm}[\cite{NOP2}]\label{thm:conelaman}
Let $(G,p)$ be a generic framework on $\C(1)$. Then $(G,p)$ is
minimally rigid if and only if $G$ is  a complete graph on at most
four vertices or $G$ is $(2,1)$-tight.
\end{thm}

 It remains an open problem to characterise generic minimally rigid frameworks on $\E$. (The natural analogue of the above theorems is known to be false.)

The final result of this section gives necessary conditions for generic global rigidity of frameworks on $\F$ which are analogous to Hendrickson's conditions for $\bR^d$.

\begin{thm}[\cite{JMN}]\label{thm:redundant}
Suppose $(G,p)$ is a generic globally rigid framework on $\F$ with $n\geq 7-\ell$ vertices.
Then $(G,p)$ is redundantly rigid on $\F$, and $G$ is k-connected,
where $k=2$ if $\F\in \{\Y,\C\}$ and $k=1$ if $\F=\E$.
\end{thm}

We believe that these necessary conditions for generic global rigidity are also sufficient when $\F\in \{\Y,\C\}$, see \cite[Conjecture $9.1$]{JMN}. One motivation for the current paper is to try to verify this conjecture by using the same proof technique as Theorem \ref{thm:2dcomb}. We will return to this in Section
\ref{sec:cyl}.

\section{The Rigidity Map}
\label{sec:gen}

We assume henceforth that $G=(V,E)$ is a graph with $V=\{v_1,v_2,\ldots,v_n\}$ and $E=\{e_1,e_2,\ldots,e_m\}$.
The {\em rigidity map} $F^G:\bR^{3n}\rightarrow \bR^m$ is defined by $F^G(p)=(\|e_1 \|^2, \dots, \|e_m\|^2)$ where $\|e_i\|^2=\|p(v_j)-p(v_k)\|^2$
when $e_i=v_jv_k$. Its differential at
the point $p$ is the map  $dF^G_p:\bR^{3n}\rightarrow \bR^m$ defined
by $dF^G_p(q)=2R(G,p)\cdot q$ where $R(G,p)$ is
 the $|E|\times 3|V|$ matrix  with rows indexed by $E$ and 3-tuples of columns indexed by $V$ in which, for
$e=v_iv_j\in E$, the submatrices in row $e$ and columns $v_i$ and $v_j$ are $p(v_i)-p(v_j)$ and  $p(v_j)-p(v_i)$, respectively, and
all other entries are zero.
We refer to $R(G,p)$ as the {\em rigidity matrix} for $(G,p)$.

We next define a rigidity map and matrix for a framework $(G,p)$ constrained to lie on our surface $\F$.
Let $\Theta^\F:\bR^{3n}\rightarrow \bR^n$ be the map defined by $\Theta^\F(p)=(h_1(p(v_1)),\dots, h_n(p(v_n)))$ where, for each $1 \leq i\leq n$,
\begin{equation}\label{eqn:surface}
h_i(x,y,z)= \begin{cases}
 x^2+y^2-r_i, & \text{if }\F=\Y(r_1,r_2,\ldots,r_n); \\
 x^2+y^2-r_iz^2, & \text{if }\F=\C(r_1,r_2,\ldots,r_n); \\
  x^2+\alpha y^2+\beta z^2-r_i, & \text{if }\F=\E(\alpha,\beta,r_1,r_2,\ldots,r_n).
\end{cases}
\end{equation}
Then the differential of $\Theta^\F$ at the
point $p$ is the map  $d\Theta^\F_p:\bR^{3n}\rightarrow \bR^n$ defined by
$d\Theta^\F_p(q)=2S(G,p)\cdot q$ where
\[ S(G,p) = \begin{bmatrix} s_1 & 0 & \dots  &0  \\  0 & s_2 & \dots &0  \\ \vdots && \ddots & \vdots\\ 0 & 0 & \dots & s_n \end{bmatrix},
\]
$s_i=s_i(p(v_i))$
and
\begin{equation}\label{eq:si}
s_i(x,y,z)= \begin{cases}
 (x,y,0), & \text{if } \F=\Y; \\
 (x,y,-r_iz), & \text{if }\F=\C;\\
 (x,\alpha y,\beta z), & \text{if }\F=\E.
\end{cases}
\end{equation}
It follows that $\rank d\Theta^\F_p=n$ if $p\in \W=\F_1\times\F_2\times\ldots\times\F_n$ and $p(v_i)\neq (0,0,0)$ for all $1\leq i \leq n$. Hence $p\in \W$ is a regular point of $\Theta^\F$ unless $\F=\C$ and $p(v_i)= (0,0,0)$ for some $1\leq i \leq n$.

The
{\em $\F$-rigidity map} $F^{G,\F}:\bR^{3n}\rightarrow \bR^{m+n}$ is defined by $F^{G,\F}=(F^G,\Theta^\F)$.
The {\em rigidity matrix}
$$R_{\F}(G,p)=\begin{bmatrix}R(G,p) \\S(G,p)\end{bmatrix}\,$$
for the framework $(G,p)$ on $\F$
is (up to scaling)  the Jacobian matrix
of $F^{G,\F}$ evaluated at the point $p$.
It is shown in \cite{NOP} that the null space of $R_{\F}(G,p)$ is the space of infinitesimal flexes of $(G,p)$ on $\F$. This allows us to
characterise infinitesimal rigidity in terms of $R_{\F}(G,p)$.

\begin{thm}[\cite{NOP}]\label{thm:surfacerank}
Let 
$(G,p)$ be a framework on $\F$.
Then $(G,p)$ is infinitesimally rigid on $\F$ if and only if $\rank R_{\F}(G,p) = 3n-\ell$.
\end{thm}


Theorem \ref{thm:surfacerank} implies that the (redundant) rigidity of a generic framework $(G,p)$ on $\F$ depends only on the graph $G$. Hence we say that $G$ is {\em (redundantly) rigid on $\F$} if some, or equivalently every, generic realisation of $G$ on $\F$ is (redundantly) rigid.

We close this section by pointing out that Theorem
\ref{thm:cylinderlaman}  implies that a graph which is (redundantly)
rigid on some family of concentric cylinders, is (redundantly) rigid
on all families of concentric cylinders irrespective of their radii.
We do not know if analogous results hold for families of concentric
cones or ellipsoids.

\section{Stresses and stress matrices}
\label{sec:stress}

In this section we develop the notion of an equilibrium stress
for a framework on our surface $\F$
and show that if $(G,p)$ is `fully realised' on $\F$ and has a maximum rank positive semi-definite stress matrix then every equivalent framework on $\F$ is an affine image of $(G,p)$.

A \emph{stress} for a framework $(G,p)$ on $\F$ is a pair $(\omega,\lambda)$, where $\omega:E\to \bR$ and $\lambda:V\to \bR$. A stress $(\omega,\lambda)$ is an \emph{equilibrium stress} if it belongs to the cokernel of $R_{\F}(G,p)$. Thus
$(\omega,\lambda)$ is an equilibrium stress for $(G,p)$ on $\F$ if and only if
 \begin{equation}\label{eq:stressdefn}
 \sum_{j=1}^n \omega_{ij}(p(v_i)-p(v_j)) + \lambda_i s_i(p(v_i))=0 \mbox{ for all $1\leq i \leq n$},
\end{equation}
where $s_i(p(v_i))$ is as defined in Equation (\ref{eq:si}), $\omega_{ij}$ is taken to be equal to $\omega_e$ if $e=v_iv_j\in E$ and to be equal to $0$ if $v_iv_j\not\in E$.
We can think of $\omega$ as a weight function on the edges
and $\lambda$ as a weight function on the vertices.
Note that, if the rows of $R_\F(G,p)$ are linearly independent, then the only equilibrium stress for $(G,p)$ is the all-zero equilibrium stress.

Given a stress $(\omega,\lambda)$ for a framework $(G,p)$ on $\F$ we define: $\Omega=\Omega(\omega)$ to be the $n\times n$ symmetric matrix with off-diagonal entries $-\omega_{ij}$ and diagonal entries $\sum_j \omega_{ij}$;
$\Lambda=\Lambda(\lambda)$ to be the $n\times n$ diagonal matrix with diagonal entries $\lambda_1,\lambda_2,\ldots,\lambda_n$; and $\Delta=\Delta(\lambda, r)$ to be the $n \times n$ diagonal matrix with diagonal entries $\lambda_1 r_1,\lambda_2 r_2,\ldots,\lambda_n r_n$.
The {\em stress matrix} associated to $(\omega,\lambda)$ on $\F$ is the $3n\times 3n$ symmetric matrix

$$\Omega_{\F}(\omega,\lambda)=\begin{bmatrix} \Omega + \Lambda & 0 & 0\\ 0 & \Gamma & 0\\ 0 & 0 & \Sigma \end{bmatrix}$$
where: $\Gamma = \Omega + \Lambda$ if $\F\in \{\Y,\C\}$ and $\Gamma=\Omega+\alpha \Lambda$ if $\F=\E$; $\Sigma=\Omega$ if $\F=\Y$, $\Sigma = \Omega-\Delta$ if $\F=\C$, and $\Sigma=\Omega+\beta \Lambda$ if $\F=\E$.
Our next result, which follows immediately from the definition of an equilibrium stress, tells us how we can use $\Omega_{\F}(\omega,\lambda)$ to determine if $(\omega,\lambda)$ is an equilibrium stress for $(G,p)$ on $\F$.

\begin{lem}
Let $(G,p)$ be a framework on $\F$ with $p(v_i)=(x_i,y_i,z_i)$ and let
\[ \Pi=\begin{bmatrix} x_1 & \dots & x_n & 0 & \dots &0 &0 &\ldots & 0\\
0 & \dots & 0 & y_1 & \dots & y_n & 0 & \dots & 0 \\
0 & \dots &0 &0 &\ldots & 0 & z_1 & \dots & z_n \end{bmatrix}.\]
Then $(\omega,\lambda)$ is an equilibrium stress for $(G,p)$ on $\F$ if and only if
$\Pi \, \Omega_{\F}=0.$
\end{lem}

We next define the {\em configuration matrix} $C_\F(G,p)$ for a framework $(G,p)$ on $\F$ by modifying the above matrix $\Pi$ as follows:

$$ C_\F(G,p)=
\begin{bmatrix} x_1 & \dots & x_n & 0 & \dots &0 &0 &\ldots & 0\\
0 & \dots & 0 & y_1 & \dots & y_n & 0 & \dots & 0 \\
0 & \dots &0 &0 &\ldots & 0 & z_1 & \dots & z_n \\
y_1 & \dots & y_n & 0 & \dots &0 &0 &\ldots & 0\\
0 & \dots & 0 & x_1 & \dots & x_n & 0 & \dots & 0 \\
0 & \dots &0&0&\ldots& 0 & 1 & \dots & 1 \end{bmatrix} \;\;\; \mbox{if } \M=\Y,$$

$$ C_\F(G,p)=
\begin{bmatrix} x_1 & \dots & x_n & 0 & \dots &0 &0 &\ldots & 0\\
0 & \dots & 0 & y_1 & \dots & y_n & 0 & \dots & 0 \\
0 & \dots &0 &0 &\ldots & 0 & z_1 & \dots & z_n \\
y_1 & \dots & y_n & 0 & \dots &0 &0 &\ldots & 0\\
0 & \dots & 0 & x_1 & \dots & x_n & 0 & \dots & 0 \\
\end{bmatrix}
\;\;\; \mbox{if } \M=\C,$$
and $C_\F(G,p)=\Pi$ if $\M=\E$.
We can use the configuration matrix to obtain an upper bound on the rank of a stress matrix.

\begin{lem}\label{lem:coker} Let $(\omega,\lambda)$ be an equilibrium stress for a framework $(G,p)$ on $\F$.
Then each row of $C_{\F}(G,p)$ belongs to the cokernel of
$\Omega_{\F}(\omega,\lambda)$,
$\rank \Omega_{\F}(\omega,\lambda) \leq 3n-\rank C_{\F}(G,p)$ and, if equality holds, then the rows of $C_{\F}(G,p)$ span the cokernel of
$\Omega_{\F}(\omega,\lambda)$.
\end{lem}

\begin{proof}
Equation (\ref{eq:stressdefn}) and the definitions of $\Omega_{\F}(\omega,\lambda)$ and
$C_{\F}(G,p)$ imply that
\[ C_{\F}(G,p) \,  \Omega_{\F}(\omega,\lambda)=0.\]
Thus each row of $C_{\F}(G,p)$ belongs to the cokernel of
$\Omega_{\F}(\omega,\lambda)$. Hence $\dim \coker \Omega_{\F}(\omega,\lambda) \geq \rank C_{\F}(G,p)$ and we have
$\rank \Omega_{\F}(\omega,\lambda) =3n-\dim \coker \Omega_{\F}(\omega,\lambda)\leq 3n-\rank C_{\F}(G,p).$
Furthermore, if equality holds, then $\coker \Omega_{\F}(\omega,\lambda)$ is equal to the row space of $C_{\F}(G,p)$.
\end{proof}

We next use Lemma \ref{lem:coker} to show that, if a framework
$(G,p)$ on $\F$ has an equilibrium stress $(\omega,\lambda)$ whose associated stress
matrix has maximum rank, then every framework $(G,q)$ on $\F$ which
has $(\omega,\lambda)$ as an equilibrium stress can be obtained from $(G,p)$ by an affine transformation.

\begin{lem}\label{lem:1}
Let $(G,p)$ and $(G,q)$ be frameworks on $\F$
and let $(\omega,\lambda)$ be an equilibrium stress for both $(G,p)$ and $(G,q)$. Suppose that $\rank \Omega_{\F}(\omega,\lambda) = 3n-\rank C_{\F}(G,p)$.
Then, for some fixed $a,b,c,d,e,f \in \bR$, we have
\begin{equation}\label{eqn:affine}q(v_i)=
\begin{bmatrix} a & b & 0\\ c & d & 0\\ 0 & 0 & e \end{bmatrix} \, p(v_i) + \begin{bmatrix} 0 \\ 0 \\ f\end{bmatrix} \mbox{ for all $1\leq i\leq n$,}\end{equation}
where $f=0$ if $\F\in\{\C,\E\}$ and $b=c=0$ if $\F=\E$.
\end{lem}

\begin{proof}
Lemma \ref{lem:coker} implies that the rows of $C_{\F}(G,p)$ span
the cokernel of $\Omega_{\F}(\omega,\lambda)$, and that each row of $C_{\F}(G,q)$ belongs to the cokernel of $\Omega_{\F}(\omega,\lambda)$.
It follows that each row of $C_{\F}(G,q)$ is a linear combination of the rows of $C_{\F}(G,p)$. The lemma now follows from the structure of the matrices
$C_{\F}(G,p)$ and $C_{\F}(G,q)$.
\end{proof}

We will say that $(G,q)$ is {\em an $\F$-affine image of $(G,p)$} if it satisfies the conclusion of Lemma \ref{lem:1}.
Our next result gives a converse to Lemma \ref{lem:1}.

\begin{lem}\label{lem:converse}
Let $(G,p)$ and $(G,q)$ be frameworks on $\F$ such that $(G,q)$ is an $\F$-affine image of $(G,p)$.
Then every equilibrium stress $(\omega,\lambda)$ for $(G,p)$ is an equilibrium stress for $(G,q)$.
\end{lem}

\begin{proof}
Since $(G,q)$ is an $\F$-affine image of $(G,p)$, we have $q(v_i)=Ap(v_i)+t$ for some fixed $A,t$ satisfying the conclusion of
Lemma \ref{lem:1}, and all $1\leq i\leq n$.
Hence
\begin{eqnarray*}
\sum_j \omega_{ij}(q(v_i)-q(v_j))+\lambda_i s_i(q(v_i)) &=& \sum_j \omega_{ij}A(p(v_i)-p(v_j)) + \lambda_i s_i(Ap(v_i)+t)
\\ &=& A\left( \sum_j \omega_{ij}(p(v_i)-p(v_j)) + \lambda_i s_i(p(v_i))\right), \end{eqnarray*}
since $s_i(Ap(v_i)+t)=As_i(p(v_i))$ by the definitions of $s_i,A,t$. The
lemma now follows by applying Equation (\ref{eq:stressdefn}).
\end{proof}

A framework $(G,p)$ on $\F$ is {\em fully realised} on $\F$ if the rows of its configuration matrix are linearly independent i.e. we have $\rank C_\F(G,p)=\mu$ where
\begin{equation}\label{eq:mu}
\mu=\begin{cases}
6 & \mbox{ if } \F=\Y; \\
5 & \mbox{ if } \F=\C; \\
3 & \mbox{ if } \F=\E.
\end{cases}
\end{equation}
It can be seen that $(G,p)$ is fully realised on $\F$ if and only if
its points do not all lie on: a plane containing or
perpendicular to the $z$-axis when $\F=\Y$; a plane containing the $z$-axis when $\F=\C$;  one of the planes $x=0$, $y=0$ or $z=0$ when $\F=\E$.

We will next use a similar argument to that used by Connelly in \cite{C82}
to show that, if $(G,p)$ has a positive semi-definite stress matrix of maximum rank then any equivalent framework is an $S$-affine image of $(G,p)$.

The {\em energy function} associated to a stress $(\omega,\lambda)$ for a framework $(G,q)$ and a family of concentric surfaces $\F$ is
defined as
\[ E_{\omega,\lambda,\F}(q)=\sum_{1\leq i<j\leq n} \omega_{ij}\|q(v_i)-q(v_j)\|^2 + \sum_{i=1}^n \lambda_i k(q(v_i))\]
where
\begin{equation}\label{eq:k}
k(x,y,z)=\begin{cases}
x^2+y^2 & \mbox{ if } \F=\Y; \\
x^2+y^2-r_iz^2 & \mbox{ if } \F=\C; \\
x^2+\alpha y^2+\beta z^2 & \mbox{ if } \F=\E.
\end{cases}
\end{equation}
Then the differential of $E_{\omega,\lambda,\F}(q)$ at a point $q$
with $q(v_i)=(x_i,y_i,z_i)$ for all $1\leq i\leq n$ is given by
\begin{equation}\label{eq:grad1}
dE_{\omega,\lambda,\F}|_q = 2
(x_1,\ldots, x_n,y_1,\dots, y_n, z_1, \dots,
z_n)\Omega_{\F}\,(\omega,\lambda).
\end{equation}
Hence, when $(G,q)$ is a framework on $\F$, $q$ is a critical point of $E_{\omega,\lambda,\F}$ if and only if $(\omega,\lambda)$ is an equilibrium stress for $(G,q)$ on $\F$.

\begin{lem}\label{lem:0}
Suppose $q\in \bR^{3n}$. If $q$ is a critical point of
$E_{\omega,\lambda,\F}$ then $E_{\omega,\lambda,\F}(q)=0$. In
addition, when $\Omega_{\F}(\omega,\lambda)$ is positive
semi-definite and $(G,q)$ lies on $\F$, we have
$E_{\omega,\lambda,\F}(q)=0$ if and only if $q$ is a critical point
of $E_{\omega,\lambda,\F}$.
\end{lem}

\begin{proof}
Suppose $q$ is a critical point of $E_{\omega,\lambda,\F}$. Then the
differential of $E_{\omega,\lambda,\F}(q)$ in the direction of $q$
is zero.
This implies that $E_{\omega,\lambda,\F}(tq)$ is constant for all
$t\in \bR$. We can now take $t=0$ to deduce that
$E_{\omega,\lambda,\F}(q)=E_{\omega,\lambda,\F}(0)=0$.

Observe that, if $q(v_i)=(x_i,y_i,z_i)$ for all $1\leq i\leq n$,
then
$$E_{\omega,\lambda,\F}(q)=(x_1,\ldots, x_n,y_1,\dots, y_n, z_1, \dots,
z_n) \Omega_{\F}(\omega,\lambda)(x_1,\ldots, x_n,y_1,\dots, y_n,
z_1, \dots, z_n)^T.$$ Thus, when $\Omega_{\F}(\omega,\lambda)$ is
positive semi-definite,
we have $E_{\omega,\lambda,\F}(q)\geq 0$ for all $q \in \bR^{3n}$.
Hence $q$ is a critical point of $E_{\omega,\lambda,\F}$ if
$E_{\omega,\lambda,\F}(q)=0$.
\end{proof}

We can now deduce that equivalent frameworks with maximum rank
positive semi-definite stress matrices are linked by affine
transformations.

\begin{thm}\label{thm:psdequiv}
Let $(G,p)$ be a framework which is fully realised on $\F$
and let $(\omega,\lambda)$ be an equilibrium stress for $(G,p)$. Suppose that $\Omega_{\F}(\omega,\lambda)$ is positive semi-definite and $\rank \Omega_{\F}(\omega,\lambda) = 3n-\mu$. Let $(G,q)$ be a framework on $\F$ which is equivalent to $(G,p)$. Then $(\omega,\lambda)$ is an equilibrium stress for $(G,q)$, and $(G,q)$ is an $\F$-affine image of $(G,p)$.
\end{thm}

\begin{proof}
Since $(\omega,\lambda)$ is an equilibrium stress for $(G,p)$ we have $E_{\omega,\lambda,\F}(p)=0$.
Then
$$ E_{\omega,\lambda,\F}(q) =  E_{\omega,\lambda,\F}({q}) -E_{\omega,\lambda,\F}(p) =
\sum_{i=1}^n \lambda_i [k(q(v_i)) - k(p(v_i))]=0$$ since
$(G,p)$ and $(G,q)$ are equivalent and both lie on $\F$. Lemma
\ref{lem:0} now implies that
$q$ is a critical point of $E_{\omega,\lambda,\F}$ and hence
$(\omega,\lambda)$ is an equilibrium stress for $(G,q)$. The last part of the theorem now
follows from Lemma \ref{lem:1}.
\end{proof}


We close this section by showing that any two equivalent generic frameworks on $\F$ which are linked by an $\F$-affine map, are in fact congruent.  

We say that a framework $(G,p)$ on $\F$ is
\emph{quasi-generic} if it is congruent to a generic framework on
$\F$.  The framework $(G,p)$  is said to be in \emph{standard position} on
$\F$ if $p(v_1)=(x_1,y_1,z_1)$ and: $p(v_1)=(0,y_1,0)$ when $\F=\Y$; $p(v_1)=(0,y_1,z_1)$ when $\F=\C$. All frameworks on $\E$ are taken to be
in standard position. Two frameworks on $\F$ are {\em
$\F$-congruent} if there is an isometry of $\F$ which maps one on to
the other.
We use $\overline{\bK}$ to denote the algebraic closure of a field $\bK$.

We will need the following result, \cite[Lemma $8$]{JMN}.

\begin{lem}\label{lem:q-g<=>}
Suppose $(G,p)$ and $(G,p_0)$ are $\F$-congruent frameworks on $\F$ and $(G,p_0)$ is in standard position on $\F$.
Then $(G,p)$ is quasi-generic if and only if $\td[\bQ(r,p_0):\bQ(r)]=2n-\ell$.
\end{lem}


\begin{lem}\label{lem:global}
Let $(G,p)$ be a generic framework on $\F$ with $n\geq 5$ vertices.
Suppose that $(G,q)$ is an equivalent framework to $(G,p)$ on $S$ which is also an $\F$-affine image of $(G,p)$.
Then $(G,q)$ is congruent to $(G,p)$
\end{lem}

\begin{proof}
We may use the isometries of
$\F$ to move $(G,p)$ and $(G,q)$ to two frameworks $(G,p_0)$ and
$(G,q_0)$ in standard position on $\F$. Then $(G,q_0)$ will be an
$\F$-affine image of $(G,p_0)$.
Let $p_0(v_i)=(x_i,y_i,z_i)$ and $q_0(v_i)=(\hat{x}_i,\hat{y}_i,\hat{z}_i)$.
We will analyse each choice of $\F$ in turn.

\textbf{Case 1}: $\F=\Y$. We have
\begin{equation}\label{eqn:affine2}q_0(v_i)=
\begin{bmatrix} a & b & 0\\ c & d & 0\\ 0 & 0 & e \end{bmatrix} \cdot p_0(v_i) + \begin{bmatrix} 0 \\ 0 \\ f\end{bmatrix} \mbox{ for all $1\leq i\leq n$.}\end{equation}
Applying Equation (\ref{eqn:affine2}) with $q_0(v_1)=(0,y_1,0)=p_0(v_1)$ ($\hat{y}_1=y_1$ since $(G,p)$ and $(G,q)$ are on $S$) reveals that $b=0=f$ and $d=1$.
For $i=2,3,\dots,n$, Equation (\ref{eqn:affine2}) now gives
\[ \begin{bmatrix} \hat{x}_i \\ \hat{y}_i \\ \hat{z}_i\end{bmatrix}=q_0(v_i)=\begin{bmatrix} a & 0 & 0\\ c & 1 & 0\\ 0 & 0 & e\end{bmatrix}\begin{bmatrix} x_i \\ y_i \\ z_i\end{bmatrix} = \begin{bmatrix} ax_i \\ cx_i+y_i \\ ez_i\end{bmatrix}. \]
Using the fact $q_0(v_i)$ and $p_0(v_i)$ are on $\Y_i$ we deduce
that
\begin{equation}\label{eqn:1}
(a^2-1+c^2)x_i^2+2cx_iy_i=0.
\end{equation}

Suppose $c\neq 0$. Then we have
\begin{equation*}
y_i=\frac{(1-a^2-c^2)x_i}{2c},
\end{equation*}
and
\[ r_i=x_i^2 + y_i^2= x_i^2+ \frac{(1-a^2-c^2)^2x_i^2}{4c^2}. \]
These equations imply that $x_i,y_i \in \overline{\bQ(r,a,c)}$. We may
now deduce that
$$\td [\bQ(r,p_0):\bQ(r)]\leq \td [\overline{\bQ(r,z_2,z_3,\dots,z_n,a,c)}:\bQ(r)]\leq n+2.$$
Since $n\geq 5$, this contradicts the fact that $\td
[\bQ(r,p_0):\bQ(r)]=2n-2$
by Lemma \ref{lem:q-g<=>}. Hence $c=0$.

Equation (\ref{eqn:1}) and the fact that $c=0$ implies $a=\pm 1$. It remains to show that $e=\pm 1$.
We may assume, without loss of generality, that $v_1v_2 \in E$. Then
\begin{eqnarray*} x_2^2+(y_1-y_2)^2 +z_2^2 &=& \|(0,y_1,0)-(x_2,y_2,z_2)\|^2 = \| p_0(v_1)-p_0(v_2)\|^2 \\ &=& \| q_0(v_1)-q_0(v_2)\|^2 = \|(0,y_1,0)-(\hat{x}_2,\hat{y}_2,\hat{z}_2)\|^2 \\ &=& \|(0,y_1,0)-A(x_2,y_2,z_2)\|^2 = \| (0,y_1,0)-(\pm x_2,y_2,ez_2)\|^2 \\ &=& x_2^2 + (y_1-y_2)^2 +e^2z_2^2. \end{eqnarray*}
Hence $z_2^2=e^2z_2^2$ and $e=\pm 1$.

We have shown that, if $q_0\neq p_0$, then $(G,q_0)$ is a reflection of $(G,p_0)$ in a plane which contains $(0,y_1,0)$ and either contains, or is perpendicular to, the $z$-axis or a composition thereof.
Hence $(G,p_0)$ and $(G,q_0)$ are congruent. This implies that $(G,p)$ and $(G,q)$ are congruent.
\newline

\textbf{Case 2}: $\F=\C$. We have
\begin{equation}\label{eqn:affine23}q_0(v_i)=
\begin{bmatrix} a & b & 0\\ c & d & 0\\ 0 & 0 & e \end{bmatrix} \cdot p_0(v_i) \mbox{ for all $1\leq i\leq n$.}\end{equation}
Since $p_0(v_1)=(0,y_1,z_1)$, $q_0(v_1)=(0,\hat{y}_1,\hat{z}_1)$, $y_1^2=r_1z_1^2$ and $\hat{y}_1^2=r_1\hat{z}_1^2$ applying Equation (\ref{eqn:affine23}) shows that $b=0$ and $d=e$. For $i=2,3,\dots,n$, we have

\[ \begin{bmatrix} \hat{x}_i \\ \hat{y}_i \\ \hat{z}_i\end{bmatrix}=q_o(v_i)=\begin{bmatrix} a & 0 & 0\\ c & d & 0\\ 0 & 0 & d\end{bmatrix}\begin{bmatrix} x_i \\ y_i \\ z_i\end{bmatrix} = \begin{bmatrix} ax_i \\ cx_i+dy_i \\ dz_i\end{bmatrix}. \]
Using the fact $q_0(v_i)$ and $p_0(v_i)$ are on $\C_i$ we deduce
that
\begin{equation}\label{eqn:12}
(a^2-1+c^2)x_i^2+2cdx_iy_i + (d^2-1)y_i^2-r_i(d^2-1)z_i^2=0.
\end{equation}

Suppose $d^2\neq 1$. Then
\begin{equation*}
z_i^2=\frac{(a^2-1+c^2)x_i^2+2cdx_iy_i + (d^2-1)y_i^2}{r_i(d^2-1)}.
\end{equation*}
Since
\[ x_i^2 + y_i^2= r_iz_i^2= \frac{(a^2-1+c^2)x_i^2+2cdx_iy_i + (d^2-1)y_i^2}{d^2-1},\]
this implies $x_i,z_i \in \overline{\bQ(r,a,c,d,y_i)}$.
We may now deduce that
$$\td [\bQ(r,p_0):\bQ(r)]\leq \td [\overline{\bQ(r,y_1,y_2,y_3,\dots,y_n,a,c,d)}:\bQ(r)]\leq n+3.$$
Since $n\geq 5$, this contradicts the fact that $\td [\bQ(r,p_0):\bQ(r)]=2n-1$, by Lemma \ref{lem:q-g<=>}.
Hence $d^2=1$. Substituting $d^2=1$ into Equation (\ref{eqn:12}) gives
\begin{equation}\label{eqn:123}
(a^2-1+c^2)x_i^2+2cdx_iy_i=0.
\end{equation}
Similar arguments to those used in Case 1 can now be applied to
deduce $c=0$ and $a=\pm 1$.

We have shown that, if $q_0\neq p_0$, then $(G,q_0)$ is a reflection of $(G,p_0)$ in the plane containing $(0,y_1,z_1)$ and the $z$-axis, or a rotation by $\pi$ around the $x$-axis, or a composition thereof.
Hence $(G,p_0)$ and $(G,q_0)$ are congruent. This implies that $(G,p)$ and $(G,q)$ are congruent.
\newline
\textbf{Case 3}: $\F=\E$. We have
\begin{equation}\label{eqn:affine234}q_0(v_i)=
\begin{bmatrix} a & 0 & 0\\ 0 & d & 0\\ 0 & 0 & e \end{bmatrix} \cdot p_0(v_i) \mbox{ for all $1\leq i\leq n$.}\end{equation}
Since $p_0(v_i)$ and $q_0(v_i)$ both lie on $\E_i$, we have $x_i^2 + \alpha y_i^2 + \beta z_i^2 = r_i$ and $a^2 x_i^2 + \alpha d^2 y_i^2 + \beta e^2 z_i^2 = r_i$. We can eliminate $x_i^2$ from these equations to obtain
\begin{equation}\label{eqn:yyy}
r_i(a^2-1)+\alpha y_i^2(d^2-a^2) + \beta z_i^2(e^2-a^2)=0.
\end{equation}
Hence, if $d^2-a^2 \neq 0$, then $x_i,y_i \in \overline{\bQ(r,a,d,e,z_i)}$. This would imply that
\[
2n = td [\bQ(r,p_0):\bQ(r)] \leq td [\overline{\bQ(r,a,d,e,z_1,z_2,\dots, z_n)}:\bQ(r)] \leq n+3,
\]
a contradiction since $n\geq 5$. Hence $d^2=a^2$. We can deduce similarly, from Equation (\ref{eqn:yyy}), that $a^2=e^2$. Equation (\ref{eqn:yyy}) now implies that $a^2=1$.

We have shown that $(G,q_0)$ is a reflection of $(G,p_0)$ in either the plane $x=0,y=0$ or $z=0$ or a composition thereof.
Hence $(G,p_0)$ and $(G,q_0)$ are congruent. This implies that $(G,p)$ and $(G,q)$ are congruent.\end{proof}

Theorem \ref{thm:psdequiv}  and Lemma \ref{lem:global} immediately imply that a generic framework on $\F$ with a maximum rank positive semi-definite stress matrix is globally rigid. We believe that the same conclusion holds without the hypothesis that there is a positive semi-definite stress matrix.

\begin{con}\label{con:affine}
Suppose that $(G,p)$ is generic on $\F$ and that $(\omega,\lambda)$ is an equilibrium stress for $(G,p)$ with $\rank \Omega_{\F}(\omega,\lambda) = 3n-\mu$. Then $(G,p)$ is globally rigid on $\F$.
\end{con}

\section{A sufficient condition for global rigidity on families of cylinders and ellipsoids}
\label{sec:main}

We will show that Conjecture \ref{con:affine} holds when $\F\in \{\Y,\E\}$ and the parameters $r_1,r_2,\dots,r_n$ are algebraically independent over $\mathbb{Q}$. 
To do this we need to change our viewpoint from the surface $S\subset \bR^3$
to a point $p\in \bR^{3n}$.

Given a map $p:V\to \bR^{3n}$, there is a unique family of concentric surfaces $\F$ with $p(v_i)\in \F_i$
for each $\F\in \{\Y,\C,\E\}$, as long as $p(v_i)$ does not lie on the $z$-axis for all $1\leq i\leq n$ when
$\F\in \{\Y,\C\}$ and $p(v_i)\neq (0,0,0)$ for all $1\leq i\leq n$ when
$\F=\E$. We will refer to $S$ as the {\em surface induced by $p$} and denote it by $S^p$.

With this restriction we can use the following result, due to Connelly \cite{C2005},  to obtain the special case of Conjecture \ref{con:affine}.

\begin{prop}[\cite{C2005}]\label{prop:Bobs3.3}
Suppose that $f:\mathbb{R}^a\rightarrow \mathbb{R}^b$ is a function, where each coordinate is a polynomial with integer coefficients, $p\in \mathbb{R}^a$ is generic, and $f(p)=f(q)$, for some $q\in \mathbb{R}^a$. Then there are (open) neighbourhoods $N_p$ of $p$ and $N_q$ of $q$ in $\mathbb{R}^a$ and a diffeomorphism $g:N_q\rightarrow N_p$ such that for all $x\in N_q$, $f(g(x))=f(x)$, and $g(q)=p$.
\end{prop}

\begin{thm}\label{thm:global}
Suppose $p$ is a generic point in $\mathbb{R}^{3n}$ and let $\F=\F^p$ for some $\F\in \{\Y,\E\}$. Let $(\omega,\lambda)$ be an equilibrium stress for $(G,p)$ on $\F$ and let $(G,q)$ be equivalent to $(G,p)$ on $\F$. Then $(\omega,\lambda)$ is an equilibrium stress for $(G,q)$ on $\F$. Furthermore, if $\rank \Omega_{\F}(\omega,\lambda) = 3n-\mu$, then $(G,p)$ is globally rigid on $\F$.
\end{thm}

\begin{proof}
Let $F:\bR^{3n}\to \bR^{m+n}$ be the modified surface rigidity map defined by $F(p)=(F^G(p),\hat\Theta^\S(p)$ where $F^G$ is the rigidity map for $G$, $\hat\Theta^\S(p)=( k(p(v_1),\ldots,k(p(v_n))$, $k(x,y,z)=x^2+y^2$ when $\F=\C$ and $k(x,y,z)=x^2+\alpha y^2+\beta z^2$ when $\F=\E$.  By Proposition \ref{prop:Bobs3.3} there exist open
neighbourhoods $N_{p}$ of $p$ and $N_{q}$ of $q$ in $\mathbb{R}^{3n}$ and
a diffeomorphism $g:N_{q}\to N_{p}$ such that $g(q)=p$ and,
for all $q\in N_{q}$, $f(g( q))=f(q)$.
Taking differentials at
$q$, and using the fact that the Jacobian matrix of $F$ evaluated at $p$ is $2R_{\F}(G,p)$, we obtain $R_{\F}(G,q)=R_{\F}(G,p)\,D$ where $D$ is the Jacobian matrix of $g$ at $q$.
Since $(\omega,\lambda)$ is an
equilibrium stress for $(G,p)$ we have $(\omega,\lambda)\,
R_{\F}(G,q)=(\omega,\lambda)\, R_{\F}(G,p)D= 0\,D=0$. 
Hence $(\omega,\lambda)$ is an equilibrium stress of $(G,q)$.

Since $(G,p)$ is generic, it is fully realised on $\F$. We can now use Lemma \ref{lem:1} and the hypothesis that
$\rank \Omega_{\F}(\omega,\lambda) = 3n-\mu$
to deduce that
$(G,q)$ is an $\F$-affine image of $(G,p)$. Lemma \ref{lem:global} now tells us that $(G,q)$ is congruent to $(G,p)$. Hence $(G,p)$ is globally rigid.
\end{proof}

The above proof works for families of cylinders and ellipsoids because we can eliminate the parameters $r_i$ from their defining equations without changing their Jacobian matrix. This is not possible for families of cones. We will discuss this further in Section \ref{sec: Closing Remarks}.

\section{1-extensions and global rigidity}
\label{sec:1ext}

Given a graph $G$, the {\em $1$-extension operation} constructs a
new graph by first deleting an edge $v_1v_2$ and then adding a
new vertex $v_0$ and three new edges $v_0v_1,v_0v_2,v_0v_3$ for some
vertex $v_3$ distinct from $v_1,v_2$. Our aim is to show that the
property of having a maximum rank stress matrix is preserved by the
1-extension operation.

\begin{lem}\label{lem:1ext}
Suppose $(G,p)$ is a generic framework on $S$ with $n\geq 3$.
Let $G'=(V',E')$ be a 1-extension of  $G$,
obtained by deleting an edge $e=v_1v_2$ and adding a new vertex $v_0$ and new edges
$v_0v_1, v_0v_2,v_0v_3$.
Then there exists a map $q:V'\to \bR^3$ such that $\rank R_{S^q}(G',q)=\rank R_{S}(G,p)+3$.
Furthermore, if $(\omega,\lambda)$ is an equilibrium stress for $(G,p)$ on $S$ and $\omega_e\neq 0$, then
there exists an equilibrium stress $(\omega',\lambda')$ for $(G',q)$ on $S^q$ such that  $\rank
\Omega_{\F^q}(\omega',\lambda')=rank \Omega_{\F}(\omega,\lambda)+3$.
\end{lem}

\begin{proof}
Define $(G',q)$ by putting $q(v)=p(v)$ for all $v\in V$ and
$q(v_0)=\frac{1}{2}p(v_1)+\frac{1}{2}p(v_2)$.
Let $S^q$ be the surface induced
by $q$.

We first consider the framework $(G'+v_1v_2-v_0v_2,q)$ on $S^q$. Its
rigidity matrix $R$ can be constructed from $R_\F(G,p)$ by adding 3
new columns indexed by $v_0$, and 3 new rows indexed by $v_0$,
$v_0v_1$ and $v_0v_3$, respectively. Since $(p(v_1),p(v_2),p(v_3))$
is a generic point on $S(r_1)\times S(r_2)\times S(r_3)$,
the $3\times 3$ submatrix $M$ of $R$ with rows indexed by $v_0,
v_0v_1,v_0v_3$ and columns indexed by $v_0$ is
non-singular.\footnote{We can consider  $\det M$ as a polynomial in
the coordinates of $(p(v_1),p(v_2),p(v_3))$. If $\det M=0$, then
genericness would imply that this polynomial evaluates to $0$ at all
points in $S(r_1)\times S(r_2)\times S(r_3)$. It is straightforward
to show that this is not the case by finding points
$(p_1,p_2,p_3)\in S(r_1)\times S(r_2)\times S(r_3)$ at which the
polynomial is nonzero. When $\F=\Y(r)$, we can take
$p_1=(\sqrt{r_1},0,0)$, $p_2=(0,\sqrt{r_2},0)$ and
$p_3=(\sqrt{r_3},0,1)$; when $\F=\C(r)$, we can take
$p_1=(\sqrt{r_1},0,1)$, $p_2=(\sqrt{r_2},0,-1)$ and
$p_3=(0,\sqrt{r_3},1)$; and when $\F=\E(r)$, we can take
$p_1=(\sqrt{r_1},0,0)$, $p_2=(0,\frac{\sqrt{r_2}}{\sqrt{\alpha}},0)$ and
$p_3=(0,0,\frac{\sqrt{r_3}}{\sqrt{\beta}})$.}
The fact that the new columns contain zeros everywhere except in the new rows
now gives $\rank R=\rank R_\F(G,p)+3$.
Since $q(v_0),q(v_1)$ and $q(v_2)$ are collinear, the rows in $R_{\F^q}(G'+v_1v_2,q)$ corresponding to
$v_0v_1,v_0v_2,v_1v_2$ are a minimal linearly dependent set. Thus
$$\rank R_{S^q}(G',q)=\rank R_{S^q}(G'+v_1v_2,q)=\rank R=\rank R_{S}(G,p)+3.$$

Let $({\omega'},{\lambda'})$ be the stress for $({G'},{q})$ on
${\F^q}$ defined by putting $\omega'_f=\omega_f$ for all $f \in
E-e$, ${\omega'}(v_1v_{0})=2\omega_e$,
${\omega'}(v_2v_{0})=2\omega_e$, ${\lambda'}(v)=\lambda(v)$ for all
$v \in V$ and ${\lambda'}(v_{0})=0$. It is straightforward to verify
that $({\omega'},{\lambda'})$ is an equilibrium stress for
$({G'},{q})$ on ${\F^q}$. Let $\omega_{ij}$ be the $ij$-th entry of
$\Omega(\omega')$ for $i\neq j$ and $\lambda_i$ be the $ii$-th entry
of $\Lambda(\lambda')$.

We first consider $\Omega({\omega'}) + \Lambda({\lambda'})$. We have

\[ \Omega({\omega'})+ \Lambda({\lambda'})=\begin{bmatrix}
4\omega_{12} & -2\omega_{12} & -2\omega_{12} & 0 & \dots & 0\\
-2\omega_{12} & \sum_j \omega_{1j} + \omega_{12}+\lambda_1 & 0 & -\omega_{13} & \dots & -\omega_{1n}\\
-2\omega_{12} & 0 & \sum_j \omega_{2j} + \omega_{12}+\lambda_2 &
-\omega_{23} & \dots & -\omega_{2n} \\ 0 & -\omega_{13} &
-\omega_{23} &  \dots & \dots & -\omega_{3n} \\ \vdots & \vdots &
\vdots &  &  & \vdots
\end{bmatrix}. \] By adding $1/2$ times the first row to
the second and third rows, respectively, this reduces to
\[\begin{bmatrix}
4\omega_{12} & -2\omega_{12} & -2\omega_{12} & 0 & \dots & 0\\
0 & \sum_j \omega_{1j}+\lambda_1 & -\omega_{12} & -\omega_{13} &
\dots & -\omega_{1n}\\ 0 & -\omega_{12} & \sum_j
\omega_{2j}+\lambda_2 & -\omega_{23} & \dots & -\omega_{2n} \\
\vdots & \vdots & \vdots & \vdots & & \vdots \\  \end{bmatrix}. \]
Now adding $1/2$ times the first column to the second and third
columns, respectively, gives
\[\begin{bmatrix}
4\omega_{12} & 0 & 0 & 0 & \dots & 0\\0 & \sum_j
\omega_{1j}+\lambda_1 & -\omega_{12} & -\omega_{13} & \dots &
-\omega_{1n}\\ 0 & -\omega_{12} & \sum_j \omega_{2j}+\lambda_2 &
-\omega_{23} & \dots & -\omega_{2n} \\ \vdots & \vdots & \vdots &
\vdots & & \vdots \\ \end{bmatrix} =
\begin{bmatrix} 4\omega_{12} & 0 & \dots & 0\\ 0 & \\ \vdots & & \Omega(\omega) + \Lambda(\lambda)\\ 0 \end{bmatrix}. \]
Since $\omega_{12}\neq 0$, we have
 $\rank \Omega({\omega'})+ \Lambda({\lambda'})=\rank \Omega(\omega)+ \Lambda(\lambda)+1$.

Since $\lambda'(v_0)=0$, we can repeat the above
argument for $\Omega({\omega'})$ when $\F=\Y$, for $\Omega({\omega'})-\Delta({\lambda'})$
when $\F=\C$, and for both $\Omega({\omega'})+\alpha\Lambda({\lambda'})$ and
$\Omega({\omega'})+\beta\Lambda({\lambda'})$ when $\F=\E$, to deduce that
$\rank \Omega_{{\F^q}}({\omega'},{\lambda'})=  \rank \Omega_{\F}(\omega,\lambda)+3$.
\end{proof}

We do not know whether we can find a framework $(G',q)$ which satisfies the conclusions of Lemma \ref{lem:1ext} and in addition has $\F^q=\F$.\footnote{Partial results are known for particular surfaces: there exists a framework $(G,q)$ with $\rank R_{\F^q}(G',q)=\rank R_{S}(G,p)+3$ and $\F^q=\F$ when $\F=\Y$ \cite{NOP}, and when $\F=\C(1)$ or $\F=\E(1)$ \cite{NOP2}.} Lacking such a result, we are forced
to consider frameworks on `generic surfaces' i.e. surfaces $S^q$ induced by some generic $q\in \bR^{3n}$.

\begin{lem}\label{lem:genrank} Suppose $(G,p)$ is an infinitesimally rigid framework on some surface $\F$. Then $(G,q)$ is infinitesimally rigid on $\F^q$ for all generic
$q\in \bR^{3n}$.
\end{lem}

\begin{proof}
Choose $q:V\to \bR^3$ such that $q(v_i)$ does not lie on the $z$-axis for all $1\leq i\leq n$ when
$\F\in \{\Y,\C\}$ and $q(v_i)\neq (0,0,0)$ for all $1\leq i\leq n$ when
$\F=\E$. Since $q(v_i)\in S_i^q$ for all $1\leq i \leq n$,
the $\F^q$-rigidity matrix for $(G,q)$ has the form $R_{\F^q}(G,q)=\begin{bmatrix}R(G,q) \\S(G,q)\end{bmatrix}$ where $R(G,q)$ is the ordinary rigidity matrix
of $(G,q)$,
\[ S(G,q) = \begin{bmatrix} s_1 & 0 & \dots  &0  \\  0 & s_2 & \dots &0  \\ \vdots && \ddots & \vdots\\ 0 & 0 & \dots & s_n \end{bmatrix},
\]
and
\begin{equation}\label{eq:si_new}
s_i= \begin{cases}
 (x_i,y_i,0), & \text{if } \F^q=\Y^q; \\
 (x_i,y_i,-\frac{x_i^2+y_i^2}{z_i}), & \text{if }\F^q=\C^q; \\
 (x_i,\alpha y_i,\beta z_i), & \text{if }\F^q=\E^q.
\end{cases}
\end{equation}
The expression for $s_i$ when $\F^q=\C^q$ is obtained by substituting $r_i=(x_i^2+y_i^2)/z_i^2$ into Equation (\ref{eq:si}). Since the entries in $R_{\F^q}(G,q)$ are rational functions of $q$, its rank will be maximised when $q$ is a generic point in $\bR^{3n}$.
\end{proof}

The analogous result for frameworks with a maximum rank stress matrix is not true in general. It becomes true, however, if we restrict our attention to
infinitesimally rigid frameworks.

\begin{lem}\label{lem:genstress} Suppose $(G,p)$ is an infinitesimally rigid framework on  $\F$
and $\rank \Omega_{\F}(\omega,\lambda)= 3n-\mu$ for some equilibrium stress $(\omega,\lambda)$ of $(G,p)$. Then $(G,q)$ has an equilibrium stress $(\omega',\lambda')$
on $\F^q$ with
$\rank \Omega_{\F^q}(\omega',\lambda')= 3n-\mu$
for all generic
$q\in \bR^{3n}$. In addition, if $\F\in \{\Y,\E\}$, then $(\omega',\lambda')$ can be chosen so that $w'_e\neq 0$ for all
$e\in E$. 
\end{lem}

\begin{proof} We adapt the proof technique of Connelly and Whiteley \cite[Theorem 5]{C&W}.
Choose $q:V\to \bR^{3n}$ such that $(G,q)$ is infinitesimally rigid on $\F^q$. We saw in the proof of Lemma \ref{lem:genrank} that the entries in
$R_{\F^q}(G,q)$ are rational functions of $q$. Since the space of equilibrium stresses of $(G,q)$ is the cokernel of $R_{\F^q}(G,q)$, each equilibrium stress of $(G,q)$ can be expressed as a pair of rational functions $(\omega(q,t),\lambda(q,t))$ of $q$ and $t$, where $t$ is a vector of $m-2n+\ell$ indeterminates. This implies that the entries in the corresponding stress matrix $\Omega_{\F^q}(\omega(q,t),\lambda(q,t))$ will also be rational functions of $q$ and $t$. Hence the rank of
$\Omega_{\F^q}(\omega(q,t),\lambda(q,t))$ will be maximised whenever $q,t$ is algebraically independent over $\bQ$. In particular, for any generic $q\in \bR^{3n}$, $(G,q)$ is infinitesimally rigid on $\F^q$ by Lemma \ref{lem:genrank}, and we can choose $t\in \bR^{m-2n+\ell}$ such that
 $\rank \Omega_{\F^q}(\omega(q,t),\lambda(q,t))=3n-\mu$.

Now suppose that $\F\in \{\Y,\E\}$ and that  $(\omega',\lambda')$ has been chosen such that the number of edges $e\in E$ with $\omega'_e=0$ is as small as possible. We may assume  that
$\omega'_e=0$ for some $e\in E$. Then $(\omega'|_{E-e},\lambda')$ is an equilibrium stress for $(G-e,p)$ on $\F^q$ and $\rank \Omega_{\F^q}(\omega'|_{E-e},\lambda')= 3n-\mu$. By Theorem \ref{thm:global}, $(G-e,p)$ is globally rigid on $\F^q$. In particular $(G-e,p)$ is rigid on $\F^q$. Since $p$ is generic, $(G-e,p)$ is infinitesimally rigid on $\F^q$. This implies that the row of $R_{\F^q}(G,p)$ indexed by $e$ is contained in a minimal linearly dependent set of rows of $R_{\F^q}(G,p)$. This gives us an equilibrium stress $(\hat \omega, \hat \lambda)$ for $(G,p)$ on $\F^q$ with $\hat \omega_e\neq 0$. Then
$(\omega'',\lambda'')=(\omega',\lambda')+c(\hat\omega,\hat\lambda)$ is an equilibrium stress for $(G,p)$ on $\F^q$ for any $c\in \bR$. We can now choose
a small $c>0$ so that  $\rank \Omega_{\F^q}(\omega'',\lambda'')= 3n-\mu$, and $\omega''_f\neq 0$ for all $f\in E$ with $\omega'_f\neq 0$.
This contradicts the choice of $(\omega',\lambda')$.
\end{proof}

We can now obtain our result on generic 1-extensions.

\begin{thm}\label{thm:1extgeneric}
Suppose $(G,p)$ is an infinitesimally rigid framework on $\F$, for some $\F\in \{\Y,\E\}$, and $(\omega,\lambda)$ is an equilibrium stress for $(G,p)$ with
$\rank
\Omega_{\F}(\omega,\lambda)=3n-\mu$. Let $G'=(V',E')$ be a $1$-extension of $G$ and $q:V'\to \bR^{3}$ such that
$q$ is generic in  $\bR^{3(n+1)}$.
Then $(G',q)$ is infinitesimally rigid on $S^q$ and has an equilibrium stress $(\omega',\lambda')$ with
$\rank
\Omega_{\F^q}(\omega',\lambda')=3(n+1)-\mu$.
\end{thm}

\begin{proof}
We may assume that $p$ is a generic point in $\bR^{3n}$  and that
$\omega_e\neq 0$ for all $e\in E$  by Lemmas \ref{lem:genrank} and \ref{lem:genstress}. We can now use Lemma \ref{lem:1ext} to deduce that there exists a
map $p^*:V'\to \bR^3$ such that $(G,p^*)$ is infinitesimally rigid on $\F^{p^*}$ with an equilibrium stress $(\omega^*,\lambda^*)$ for $(G',p^*)$ on $S^{p^*}$ such that  $\rank
\Omega_{\F^{p^*}}(\omega^*,\lambda^*)=3(n+1)+3$. The theorem now follows by another application of Lemmas \ref{lem:genrank} and \ref{lem:genstress}.
\end{proof}

\section{Global rigidity on concentric cylinders}
\label{sec:cyl}

In this section we apply our results to make progress on the conjectured characterisation of global rigidity on concentric cylinders given in \cite[Conjecture $9.1$]{JMN}, see also \cite[Conjecture $5.7$]{Nix}.

\begin{con}\label{con:redundant}
Suppose $(G,p)$ is a generic  framework on a family of concentric cylinders $\Y$.
Then $(G,p)$ is globally rigid if and only if $G$ is a complete graph on at most four vertices, or $G$ is 2-connected and redundantly rigid on $\Y$.
\end{con}

We have seen that the redundant rigidity of $G$ on $\Y$ is independent of the radii of the cylinders in $\Y$. Thus
Conjecture \ref{con:redundant} would imply that the global rigidity of a generic realisation of $G$ on a family of concentric cylinders is also independent of the radii of the cylinders.

Theorem \ref{thm:redundant} shows that the combinatorial conditions given in Conjecture \ref{con:redundant} are necessary for global rigidity. We could try to demonstrate sufficiency using a similar proof technique to that of Theorem \ref{thm:2dcomb}. This would involve two steps: (i) a graph theoretic step obtaining a recursive construction for 2-connected, redundantly rigid graphs; (ii) a geometric step showing that each operation used in the recursive construction preserves global rigidity. Part (i) would be resolved by the following conjecture (which uses the base graphs $K_5-e,H_1,H_2$ and the operations of $1$-, $2$- and $3$-join illustrated in Figures \ref{fig:smallgraphs} and \ref{fig:sums}).

\begin{con}\label{con:recurse}
Suppose $G$ is a $2$-connected graph which is redundantly rigid on some (or equivalently every) family of concentric cylinders. Then $G$ can be obtained from either $K_5-e$, $H_1$ or $H_2$ by recursively applying the operations of edge addition, $1$-extension, and  $1$-, $2$- and $3$-join.
\end{con}

The results of \cite{Nix} verify the special case of this conjecture when $|E|=2|V|-1$ i.e. $E$ is a circuit in the generic rigidity matroid for the cylinder.

 \begin{figure}[htp]
\begin{center}
\begin{tikzpicture}[very thick,scale=.47]
\filldraw (-.5,0) circle (3pt)node[anchor=east]{$v_5$};
\filldraw (0,3) circle (3pt)node[anchor=east]{$v_1$};
\filldraw (3.5,0) circle (3pt)node[anchor=west]{$v_3$};
\filldraw (3,3) circle (3pt)node[anchor=west]{$v_2$};
\filldraw (1.5,-1.5) circle (3pt)node[anchor=west]{$v_4$};

 \draw[black,thick]
(1.5,-1.5) -- (-.5,0) -- (0,3) -- (3.5,0) -- (3,3) -- (-.5,0) -- (3.5,0);

\draw[black,thick]
(1.5,-1.5) -- (0,3) -- (3,3) -- (1.5,-1.5);

  \node [rectangle, draw=white, fill=white] (b) at (1.5,-3) {(a)};
        \end{tikzpicture}
          \hspace{0.5cm}
     \begin{tikzpicture}[very thick,scale=.47]
\filldraw (0,0) circle (3pt)node[anchor=east]{$v_4$};
\filldraw (0,3.5) circle (3pt)node[anchor=east]{$v_1$};
\filldraw (3.5,0) circle (3pt)node[anchor=north]{$v_5$};
\filldraw (3.5,3.5) circle (3pt)node[anchor=south]{$v_2$};
\filldraw (7,0) circle (3pt)node[anchor=west]{$v_6$};
\filldraw (7,3.5) circle (3pt)node[anchor=west]{$v_3$};

 \draw[black,thick]
(0,0) -- (0,3.5) -- (3.5,0) -- (3.5,3.5) -- (0,0) -- (3.5,0) -- (7,3.5);

\draw[black,thick]
(0,3.5) -- (3.5,3.5) -- (7,3.5) -- (7,0) -- (3.5,3.5);

\draw[black,thick]
(7,0) -- (3.5,0);

\node [rectangle, draw=white, fill=white] (b) at (3.5,-2) {(b)};
\end{tikzpicture}
       \hspace{0.5cm}
    \begin{tikzpicture}[very thick,scale=.47]
\filldraw (0,-1) circle (3pt)node[anchor=east]{$v_4$};
\filldraw (-1.5,2) circle (3pt)node[anchor=east]{$v_1$};
\filldraw (3,0) circle (3pt)node[anchor=north]{$v_5$};
\filldraw (1.5,3) circle (3pt)node[anchor=south]{$v_2$};
\filldraw (6,-1) circle (3pt)node[anchor=north]{$v_6$};
\filldraw (4.5,3) circle (3pt)node[anchor=south]{$v_3$};
\filldraw (7.5,2) circle (3pt)node[anchor=west]{$v_7$};

 \draw[black,thick]
(0,-1) -- (-1.5,2) -- (3,0) -- (1.5,3) -- (0,-1) -- (3,0) -- (4.5,3);

\draw[black,thick]
(-1.5,2) -- (1.5,3) -- (4.5,3) -- (6,-1) -- (7.5,2) -- (4.5,3);

\draw[black,thick]
(6,-1) -- (3,0) -- (7.5,2);

  \node [rectangle, draw=white, fill=white] (b) at (3,-3) {(c)};
         \end{tikzpicture}
\end{center}
\vspace{-0.3cm}
\caption{The graphs $K_5-e,H_1$ and $H_2$.}
\label{fig:smallgraphs}
\end{figure}
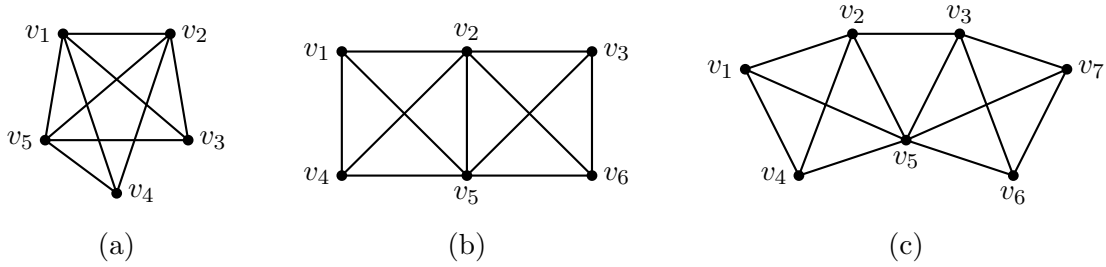

\begin{center}
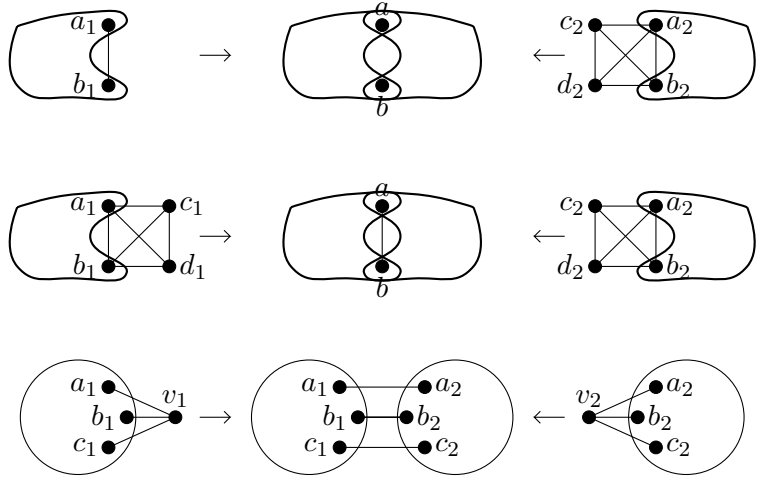
\begin{figure}[ht]
\centering
\begin{tikzpicture}[scale=0.8]

\filldraw (0,3.5) circle (3pt) node[anchor=south]{$a$};
\filldraw (0,2.5) circle (3pt) node[anchor=north]{$b$};

\filldraw (-4.5,2.5) circle (3pt) node[anchor=east]{$b_1$};
\filldraw (-4.5,3.5) circle (3pt) node[anchor=east]{$a_1$};

\filldraw (4.5,2.5) circle (3pt) node[anchor=west]{$b_2$};
\filldraw (4.5,3.5) circle (3pt) node[anchor=west]{$a_2$};
\filldraw (3.5,2.5) circle (3pt) node[anchor=east]{$d_2$};
\filldraw (3.5,3.5) circle (3pt) node[anchor=east]{$c_2$};

\draw[thick] plot[smooth, tension=1] coordinates{(-1.5,3.48) (-1.5,2.5) (-.5,2.3) (.3,2.4) (-.3,3) (.3,3.6)(-.5,3.7) (-1.5,3.48)};

\draw[thick] plot[smooth, tension=1] coordinates{(1.5,3.48) (1.5,2.5) (.5,2.3) (-.3,2.4) (.3,3) (-.3,3.6)(.5,3.7) (1.5,3.48)};

\draw[thick] plot[smooth, tension=1] coordinates{(-6,3.48) (-6,2.5) (-5,2.3) (-4.2,2.4) (-4.8,3) (-4.2,3.6)(-5,3.7) (-6,3.48)};

\draw[thick] plot[smooth, tension=1] coordinates{(6,3.48) (6,2.5) (5,2.3) (4.2,2.4) (4.8,3) (4.2,3.6)(5,3.7) (6,3.48)};

\draw[black]
(-4.5,3.5) -- (-4.5,2.5);

\draw[black]
(4.5,3.5) -- (4.5,2.5) -- (3.5,2.5) -- (3.5,3.5) -- (4.5,3.5) -- (3.5,2.5);

\draw[black]
(4.5,2.5) -- (3.5,3.5);

\draw[black]
(-3,3) -- (-2.5,3) -- (-2.6,3.1);

\draw[black]
(-2.5,3) -- (-2.6,2.9);

\draw[black]
(3,3) -- (2.5,3) -- (2.6,2.9);

\draw[black]
(2.5,3) -- (2.6,3.1);

\filldraw (0,.5) circle (3pt) node[anchor=south]{$a$};
\filldraw (0,-.5) circle (3pt) node[anchor=north]{$b$};

\filldraw (-4.5,-.5) circle (3pt) node[anchor=east]{$b_1$};
\filldraw (-4.5,.5) circle (3pt) node[anchor=east]{$a_1$};
\filldraw (-3.5,-.5) circle (3pt) node[anchor=west]{$d_1$};
\filldraw (-3.5,.5) circle (3pt) node[anchor=west]{$c_1$};

\filldraw (4.5,-.5) circle (3pt) node[anchor=west]{$b_2$};
\filldraw (4.5,.5) circle (3pt) node[anchor=west]{$a_2$};
\filldraw (3.5,-.5) circle (3pt) node[anchor=east]{$d_2$};
\filldraw (3.5,.5) circle (3pt) node[anchor=east]{$c_2$};

\draw[thick] plot[smooth, tension=1] coordinates{(-1.5,0.48) (-1.5,-.5) (-.5,-.7) (.3,-.6) (-.3,0) (.3,.6)(-.5,.7) (-1.5,.48)};

\draw[thick] plot[smooth, tension=1] coordinates{(1.5,0.48) (1.5,-.5) (.5,-.7) (-.3,-.6) (.3,0) (-.3,.6)(.5,.7) (1.5,.48)};

\draw[thick] plot[smooth, tension=1] coordinates{(-6,0.48) (-6,-.5) (-5,-.7) (-4.2,-.6) (-4.8,0) (-4.2,.6)(-5,.7) (-6,.48)};

\draw[thick] plot[smooth, tension=1] coordinates{(6,0.48) (6,-.5) (5,-.7) (4.2,-.6) (4.8,0) (4.2,.6)(5,.7) (6,.48)};

\draw[black]
(-4.5,.5) -- (-4.5,-.5) -- (-3.5,-.5) -- (-3.5,.5) -- (-4.5,.5) -- (-3.5,-.5);

\draw[black]
(-4.5,-.5) -- (-3.5,.5);

\draw[black]
(4.5,.5) -- (4.5,-.5) -- (3.5,-.5) -- (3.5,.5) -- (4.5,.5) -- (3.5,-.5);

\draw[black]
(4.5,-.5) -- (3.5,.5);

\draw[black]
(0,.5) -- (0,-.5);

\draw[black]
(-3,0) -- (-2.5,0) -- (-2.6,.1);

\draw[black]
(-2.5,0) -- (-2.6,-.1);

\draw[black]
(3,0) -- (2.5,0) -- (2.6,-.1);

\draw[black]
(2.5,0) -- (2.6,.1);

\draw (-5,-3) circle (27pt);
\draw (5,-3) circle (27pt);
\draw (1.2,-3) circle (27pt);
\draw (-1.2,-3) circle (27pt);

\filldraw (-3.4,-3) circle (3pt) node[anchor=south]{$v_1$};
\filldraw (3.4,-3) circle (3pt) node[anchor=south]{$v_2$};
\filldraw (-4.2,-3) circle (3pt) node[anchor=east]{$b_1$};
\filldraw (4.2,-3) circle (3pt) node[anchor=west]{$b_2$};
\filldraw (-4.5,-3.5) circle (3pt) node[anchor=east]{$c_1$};
\filldraw (-4.5,-2.5) circle (3pt) node[anchor=east]{$a_1$};
\filldraw (4.5,-3.5) circle (3pt) node[anchor=west]{$c_2$};
\filldraw (4.5,-2.5) circle (3pt) node[anchor=west]{$a_2$};
\filldraw (0.7,-2.5) circle (3pt) node[anchor=west]{$a_2$};
\filldraw (.7,-3.5) circle (3pt) node[anchor=west]{$c_2$};
\filldraw (.4,-3) circle (3pt) node[anchor=west]{$b_2$};
\filldraw (-.7,-2.5) circle (3pt) node[anchor=east]{$a_1$};
\filldraw (-.7,-3.5) circle (3pt) node[anchor=east]{$c_1$};
\filldraw (-.4,-3) circle (3pt) node[anchor=east]{$b_1$};

\draw[black]
(-.4,-3) -- (.4,-3);

\draw[black]
(-.7,-2.5) -- (.7,-2.5);

\draw[black]
(-.7,-3.5) -- (.7,-3.5);

\draw[black]
(-.4,-3) -- (.4,-3);

\draw[black]
(-.4,-3) -- (.4,-3);

\draw[black]
(-4.5,-3.5) -- (-3.4,-3) -- (-4.2,-3);

\draw[black]
(-3.4,-3) -- (-4.5,-2.5);

\draw[black]
(4.5,-3.5) -- (3.4,-3) -- (4.2,-3);

\draw[black]
(3.4,-3) -- (4.5,-2.5);

\draw[black]
(-3,-3) -- (-2.5,-3) -- (-2.6,-3.1);

\draw[black]
(-2.5,-3) -- (-2.6,-2.9);

\draw[black]
(3,-3) -- (2.5,-3) -- (2.6,-3.1);

\draw[black]
(2.5,-3) -- (2.6,-2.9);

\end{tikzpicture}
\caption{The 1-, 2- and 3-join operations. The 1- and 2-join operations form the graphs in the centre by merging $a_1$ and $a_2$ into $a$, and $b_1$ and $b_2$ into $b$.}
\label{fig:sums}
\end{figure}
\end{center}

We close by showing that all graphs constructed from our base graphs using the edge addition and 1-extension operations are generically globally rigid on concentric cylinders with algebraically independent radii.

\begin{thm}\label{thm:cylinderglobal1ext}
Suppose $G$ is a graph on $n$ vertices which can be constructed from $K_5-e$, $H_1$, or $H_2$ by a sequence of $1$-extensions and edge additions. Then $(G,p)$ is globally rigid on $\Y^p$ for all generic $p\in \bR^{3n}$.
\end{thm}

\begin{proof}
We use induction on $n$ to show that $(G,p)$ is infinitesimally rigid on $\Y^p$ and has an equilibrium stress $(\omega,\lambda)$ with $\rank \Omega_{\Y^p}(\omega,\lambda)=3n-6$. The result will then follow from Theorem \ref{thm:global}. The base case of the induction is when $G\in \{K_5-e,H_1,H_2\}$. We construct a particular realisation $(G,p)$ for each such $G$ which is infinitesimally rigid on $\Y^p$ and has an equilibrium stress with a full rank stress matrix in Appendix A. We may deduce that the same properties hold for all generic $p$ by applying Lemmas \ref{lem:genrank} and \ref{lem:genstress}. To complete the induction we need to show that the $1$-extension and edge addition operations preserve the  properties of
infinitesimal rigidity and having a maximum rank stress matrix. This is trivially true for edge addition. It holds for 1-extension by Theorem
\ref{thm:1extgeneric}.
\end{proof}

We conjecture that Theorem \ref{thm:cylinderglobal1ext} can be strengthened to show that, if $G$ can be constructed as in Theorem \ref{thm:cylinderglobal1ext} and $(G,p)$ is a generic framework on any family of concentric cylinders $\Y$, then $(G,p)$ is globally rigid on $\Y$.

\section{Closing Remarks}\label{sec: Closing Remarks}

\noindent 1. Conjecture \ref{con:affine} would  follow from  Lemmas \ref{lem:1} and \ref{lem:global} if we could show that equivalent generic frameworks on $\F$ must have the same equilibrium stresses. To date we have only been able to prove the following partial result.

\begin{thm}\label{thm:genstressPartial}
Let $(G,p_0)$ be a generic framework on $\F$ and $(\omega,\lambda)$ be an equilibrium stress for
$(G,p_0)$. Let $(G,q_0)$ be equivalent to $(G,p_0)$. Then $(\omega,\lambda')$ is an equilibrium stress for $(G,q_0)$ for some $\lambda' \in \mathbb{R}^n$.
\end{thm}

The proof of Theorem \ref{thm:genstressPartial} is given in Appendix B.\\

\noindent 2. It follows from \cite{C2005} and  \cite{GHT} that a generic framework in $\bR^d$ with $n\geq d+2$ vertices is globally rigid if and only if it has a stress matrix of rank
$n-d-1$.  It is conceivable that the stress matrix condition given in Theorem
 \ref{thm:global} provides a necessary, as
 well as a sufficient, condition for the global rigidity of a generic
 framework on $\F$ whenever the framework has at least $7-\ell$
 vertices.
The following examples indicate why we need this lower bound on $n$.

 The smallest redundantly rigid graph on the cone is $K_5$, but
  no framework $(K_5,p)$ on $\C$ can have a stress matrix with the maximum
  possible rank of $3n-\mu=10$. To see this consider a generic $p\in
  \bR^{15}$. Since every equilibrium stress for $(K_5,p)$ in $\bR^3$ is an equilibrium stress for $(K_5,p)$ on $\C^p$, and since the spaces of equilibrium stresses
  for $(K_5,p)$ in $\bR^3$ and on $\C^p$ are both 1-dimensional, these
  spaces are the same. This implies that every equilibrium stress
  $(\omega,\lambda)$ for $(K_5,p)$ has $\lambda=0$ and $\rank
  \Omega(\omega)\leq 3$. Hence $\rank \Omega_{\C^p}(\omega,\lambda)\leq 9$.
  On the other hand, $(K_5,p)$ is globally rigid on $\C^p$ for all $p$.

 Similarly,
 the smallest redundantly rigid graph on the ellipsoid is $K_6-\{e,f\}$ for
 two nonadjacent edges $e,f$, but no framework $(K_6-\{e,f\},p)$ on $\E$ can have
 a stress matrix with the maximum possible rank of $3n-\mu=15$. (We do not know whether every generic framework $(K_6-\{e,f\},p)$ on $\E^p$ is globally rigid.)\\

\textbf{Acknowledgements.} We would like to thank the School of Mathematics, University of Bristol for providing partial financial support for this research. We would also like to thank Lee Butler for helpful discussions concerning semi-algebraic sets and Bob Connelly for many helpful discussions.

\section*{Appendix A: Base graphs}

We define a framework $(G,p)$ for $G\in \{K_5-e, H_1,H_2\}$ which is infinitesimally rigid on $\Y^p$ and has a self-stress $(\omega,\lambda)$  on $\Y^p$ with maximum rank stress matrix.
 We will use the labeling of the vertices given in Figure \ref{fig:smallgraphs} and adopt the convention that  $\omega_{ij}$  is the weight on the edge $v_iv_j$ in $\omega$ and  $\lambda_i$  is the weight on the vertex $v_i$ in $\lambda$.
\\

\subsection*{\bf Case 1: $G=K_5-e$}
Let
$(G,p)$ and $(\omega,\lambda)$ be defined by
$p(v_1)=(0,1,0), p(v_2)=(1,1,1), p(v_3)=(-1,-2,-1), p(v_4)=(2,3,4), p(v_5)=(5,1,-1)$,
$$(\omega_{12},\omega_{13},\omega_{14},\omega_{15},\omega_{23},\omega_{24},\omega_{25},\omega_{35},\omega_{45})=
(-369,192,153,51,-96,-279,-138,32,45)$$
and
$$(\lambda_1,\lambda_2,\lambda_3,\lambda_4,\lambda_5)=(-270,-270,-192,54,-6)).$$
It is straightforward to check that $\rank R_{\Y^p}(G,p)=13$, that $(\omega,\lambda)\cdot R_{\Y^p}(G,p)=0$ and that $\rank \Omega_{\Y^p}(\omega,\lambda)=9$.

\subsection*{\bf Case 2: $G=H_1$}
Let
$(G,p)$ and $(\omega,\lambda)$ be defined by
$p(v_1)=(0,1,0), p(v_2)=(3,1,0), p(v_3)=(1,4,1),p(v_4)=(1,2,2), p(v_5)=(2,2,3), p(v_6)=(6,0,2)$,
\begin{multline*}
(\omega_{12},\omega_{13},\omega_{15},\omega_{23},\omega_{24},\omega_{25},\omega_{26},\omega_{35},\omega_{36},\omega_{45},\omega_{56})\\=
(41,-246,369,-123,30,48,60,50,-40,492,56)
\end{multline*}
and
$$(\lambda_1,\lambda_2,\lambda_3,\lambda_4,\lambda_5,\lambda_6)=(-123,-39,30,123,-102,28).$$
It is straightforward to check that $\rank R_{\Y^p}(G,p)=16$, that $(\omega,\lambda)\cdot R_{\Y^p}(G,p)=0$  and that $\rank \Omega_{\Y^p}(\omega,\lambda)=12$.

\subsection*{\bf Case 3: $G=H_2$}
Let
$(G,p)$ and $(\omega,\lambda)$ be defined by
$p(v_1)=(0,1,0), p(v_2)=(3,1,0), p(v_3)=(1,4,1), p(v_4)=(1,2,2),p(v_5)=(2,2,3), p(v_6)=(6,0,2),p(v_7)=(3,4,3)$,
\begin{multline*}
(\omega_{12},\omega_{13},\omega_{15},\omega_{23},\omega_{24},\omega_{25},\omega_{35},\omega_{36},\omega_{37},\omega_{45},\omega_{56},\omega_{57},\omega_{67})\\
=(-58,348,-522,-108,-24,-40,14,21,-696,56,588,-42)
\end{multline*}
and
$$(\lambda_1,\lambda_2,\lambda_3,\lambda_4,\lambda_5,\lambda_6,\lambda_7)=(-174,-6,24,174,372,-28,-252).$$
It is straightforward to check that $\rank R_{\Y^p}(G,p)=19$, that $(\omega,\lambda)\cdot R_{\Y^p}(G,p)=0$  and that  $\rank \Omega_{\Y^p}(\omega,\lambda)=15$.

\section*{Appendix B: Proof of Theorem \ref{thm:genstressPartial}}
\label{sec:alggeom}

First we recall some basic
concepts from differential and algebraic geometry, and prove a key
technical result, Proposition \ref{prop:bob3.3}, which extends Proposition \ref{prop:Bobs3.3} to the case when the domain of $f$ is an algebraic set.


Let $M$ be a smooth manifold and $f:M\rightarrow \bR^m$ be a smooth map. Then $x \in M$ is  a \emph{regular point of $f$} if $df|_x$ has maximum rank, and  $f(x)$ is a \emph{regular value of $f$} if, for all $y \in f^{-1}(f(x))$, $y$ is a regular point of $f$.

\begin{lem}\label{lem:2man} For $i=1,2$, let $M_i$ be an open subset of $\bR^n$, $p_i\in M_i$, and $f_i:M_i\to \bR^m$ be a smooth map with
$\rank df_i|_{p_i} = m$  and $f_1(p_1)=f_2(p_2)$. Then there exist open neighbourhoods $N_1$ of $p_1$, $N_2$ of $p_2$, and a diffeomorphism
$g:N_1\to N_2$ such that $f_2(g(x))=f_1(x)$ for all $x\in N_1$.
\end{lem}

\begin{proof} We first consider the case when $m=n$.
By the Inverse Function Theorem there exist neighbourhoods $\tilde N_i\subseteq M_i$ of $p_i$ such that $f_i$ maps $\tilde N_i$ diffeomorphically onto $f_i(\tilde N_i)$ for $i=1,2$. Let $W=f_1(\tilde N_1)\cap f_2(\tilde N_2)$ and then let $N_i=f_i^{-1}(W)$ for $i=1,2$. We have $f_1({N}_1)=W=f_2({N}_2)$. Thus we may choose $g=f_2^{-1} \circ f_1$ and find $f_2(g(x))=f_2(f_2^{-1}( f_1(x)))=f_1(x)$ for all $x\in N_1$.

We next consider the case when $m<n$. Let $F_i:M_i\rightarrow \bR^m \times \bR^{n-m}$ be defined by $F_i(x)=(f_i(x),x_{m+1},x_{m+2},\dots, x_n)$. Then $\rank dF_i|_{p_i}=n$. By the Inverse Function Theorem there exist neighbourhoods $\tilde N_i \subseteq M_i$ of $p_i$ such that $F_i$ is a diffeomorphism from $\tilde N_i$ to $F_i(\tilde N_i)\subseteq \bR^m \times \bR^{n-m}$. Let $F_i(\tilde N_i)=U_i\times V_i$ where $U_i \subseteq \bR^m$ and $V_i\subseteq \bR^{n-m}$. Then $V_i$ is diffeomorphic to $\bR^{n-m}$ for $i=1,2$ so we can choose a diffeomorphism $h:V_1\rightarrow V_2$ such that $h(\overline{p}_1)=\overline{p}_2$, where $\overline{p}_i$ is the projection of $p_i$ onto its last $n-m$ coordinates. Let $\iota$ be the identity map on $U_1$ and let $H=(\iota,h):U_1\times V_1 \rightarrow U_1\times V_2$. Let $F_1'=H\circ F_1$. Then we have $F_1'(p_1)=(f_1(p_1),h(\overline{p}_1))=(f_2(p_2),\overline{p}_2)=F_2(p_2)$. By the previous paragraph there exist neighbourhoods $N_i \subseteq \tilde N_i$ of $p_i$ and a diffeomorphism $g:N_1\rightarrow N_2\subseteq \bR^n$ such that $F_2(g(x))=F_1'(x)$ for all $x \in N_1$. Since $F_1'(x)=(f_1(x),h(\overline{x}))$ and $F_2(g(x))=(f_2(g(x)),\overline{g(x)})$ we have $f_1(x)=f_2(g(x))$ for all $x\in N_1$.
\end{proof}

Let $\bK$ be a field such that $\bQ\subseteq \bK  \subseteq \bR$.
 A set $W\subseteq \bR^n$ is an {\em algebraic set defined  over $\bK$} if
 $W=\{x\in \bR^n: P_i(x)=0 \mbox{ for all } 1 \leq i \leq n\}$ where $P_i\in \bK [X_1,\dots, X_n]$ for $1 \leq i \leq m$.
An algebraic set $W$ is \emph{irreducible} if it cannot be expressed as the union of two  algebraic proper subsets defined over $\bR$.
The {\em dimension} of $W$, $\dim W$, is the largest integer $t$ for which $W$ has an open subset homeomorphic to $\bR^t$.
A point $p\in W$ is {\em generic over $\bK $} if every $h\in \bK [X]$ satisfying
$h(p)=0$ has $h(x)=0$ for all $x\in W$.

\begin{lem}[\cite{JMN}]\label{cor:realvar}
Let $\bK $ be a field with $\bQ\subseteq \bK \subseteq \bR$, $W\subseteq \bR^n$ be an algebraic set defined over $\bK$  and $p\in W$.  Then
$\dim W\geq \td[\bK (p):\bK ]$.  Furthermore, if $W$ is irreducible  and $\dim W= \td[\bK (p):\bK ]$, then $p$ is a generic point of $W$.
\end{lem}

Note that, if $(G,p)$ is a generic framework on $\F$, then Lemma \ref{cor:realvar} implies that $p$ is a generic point of the irreducible algebraic set $\F_1\times \F_2\times \ldots \times S_n$ defined over $\bQ(r)$ in $\bR^{3n}$.

A set $A\subseteq \bR^n$ is a {\em semi-algebraic set defined over $\bK$} if it
can be expressed as a finite union of sets of the form
$$\{x\in \bR^n : \mbox{$P_i(x)=0$ for $1\leq i\leq s$  and $Q_j(x)>0$ for $1\leq j\leq t$} \},$$
where $P_i,Q_j\in \bK[X_1,\ldots,X_n]$ for $1\leq i\leq s$ and
$1\leq j\leq t$. It is easy to see that the family of semi-algebraic
sets defined over $\bK$ is closed under union and intersection. A
deeper result is that if $A\subseteq \bR^n$ is a semi-algebraic set
defined over $\bK$ and $f:A\to \bR^m$ is a map in which each
coordinate is a polynomial with coefficients in $\bK$, then $f(A)$
is a {semi-algebraic set defined over $\bK$}. Another result we
shall need is that a semi-algebraic set $A$ can be partitioned into
a finite number of semi-algebraic subsets $C_1,C_2,\ldots C_t$,
called {\em cells}, such that, for all $1\leq i\leq t$, $C_i$ is
diffeomorphic to $\bR^{m_i}$ for some integer $m_i\geq 0$ (where
$\bR^0$ is taken to be a single point). The {\em dimension} of $A$
is the largest integer $t$ for which $A$ has an open subset
homeomorphic to $\bR^t$.
The \emph{Zariski closure},  $A^*$, of $A$  is the smallest algebraic set defined over $\bR$  which contains $A$.
It is known that $A^*$ is an algebraic set defined over $\bL$, for some  finite field extension  $\bL$ of $\bK$, and that
$\dim A=\dim A^*$. We refer the reader to \cite{BPR,BCR} for more information on semi-algebraic sets.

We can now obtain our desired extension of Proposition \ref{prop:Bobs3.3}.

\begin{prop}\label{prop:bob3.3}
Let $\bK$ be a field with $\bQ\subseteq \bK \subseteq \bR$, $W\subset \bR^a$ be
an irreducible algebraic set defined
over $\bK$ of dimension $n$,
and $f:W\to \bR^b$ be a function where each coordinate is a
polynomial with coefficients in $\bK$. Suppose that the maximum rank of the differential of $f$ is $m$, and that
$p_0\in W$ with $\td[\bK(p_0):\bK]=n$. Then $\rank df|_{p_0}=m$. Furthermore, if
$q_0\in W$ and $f(p_0)=f(q_0)$, then there exist
open neighbourhoods $N_{p_0}$ of $p_0$ and $N_{q_0}$ of $q_0$ in $W$ and a
diffeomorphism $g:N_{q_0}\to N_{p_0}$ such that $g(q_0)=p_0$ and,  for all $q\in
N_{q_0}$, $f(g(q))=f(q)$.
\end{prop}

\begin{proof}
We first show that $\rank df|_{p_0}=m=\rank df|_{q_0}$, and that there exist
open neighbourhoods $M_{p_0}$ of $p_0$ and $M_{q_0}$ of $q_0$ in $W$ such that
$f(M_{p_0})=f(M_{q_0})$ and $f(M_{p_0})$ is diffeomorphic to $\bR^m$. We then complete the proof by applying Lemma \ref{lem:2man}.

By Lemma \ref{cor:realvar}, $p_0$ is a generic point of $W$. We can
now use \cite[Lemma 2.7 and Proposition 2.32]{GHT}
to deduce that $f(p_0)$ is a regular value of $f$. In particular, we have $\rank df|_{p_0}=m=\rank df|_{q_0}$.
The Constant Rank Theorem (see, for example, \cite[Theorem 9]{Spi}) now implies that we can choose
disjoint open balls $B(p_1,\epsilon)$ and $B(q_1,\delta)$ in $\bR^{a}$ such that:
$p_0\in B(p_1,\epsilon)\cap W$; $q_0\in B(q_1,\delta)\cap W$; $p_1,q_1\in \bQ^a$; $\epsilon,\delta\in \bQ$;
both $B(p_1,\epsilon)\cap W$ and $B(q_1,\delta)\cap W$ are diffeomorphic to $\bR^n$; both $f(B(p_1,\epsilon)\cap W)$ and $f(B(q_1,\delta)\cap W)$ are diffeomorphic to $\bR^m$.

Let $U_{p_0}=B(p_1,\epsilon)\cap W$ and
$U_{q_0}=B(q_1,\epsilon)\cap W$. Since $U_{p_0}$ and $U_{q_0}$ are
both semi-algebraic defined over $\bK$, $f(U_{p_0})$ and
$f(U_{q_0})$ are both semi-algebraic defined over $\bK$, and hence
$T=f(U_{p_0})\cap f(U_{q_0})$ is semi-algebraic defined over $\bK$.
The facts that $f$ is a polynomial map, $td[\bK(p_0):\bK]=n$ and
$\rank df|_{p_0}=m\leq n$ imply that $td[\bK(f(p_0)):\bK]=m$, see
for example \cite[Lemma 3.1]{JJZ}. Let $C_1,C_2,\ldots,C_t$ be a
cell decomposition of $T$ with $f(p_0)\in C_1$, and let $C_1^*$ be
the Zariski closure of $C_1$. Then $C_1^*$ is an algebraic set
defined over some finite field extension $\bL$ of $\bK$. Since
$f(p_0)\in C_1^*$, Lemma \ref{cor:realvar} gives
$$\dim C_1=\dim C_1^*\geq td[\bL(f(p_0)):\bL]=td[\bK(f(p_0)):\bK]=m.$$
Since $C_1\subseteq f(U_{p_0})$ and
$f(U_{p_0})$ is diffeomorphic to $\bR^m$, we must have $\dim C_1=m$.
We can now take $M_{p_0}=f^{-1}(C_1)\cap U_{p_0}$ and $M_{q_0}=f^{-1}(C_1)\cap U_{q_0}$.
 Then $f(M_{p_0})=C_1=f(M_{q_0})$ and $C_1$ is diffeomorphic to $\bR^m$.

The proposition now follows from
Lemma \ref{lem:2man} by choosing $M_1=M_{p_0}$, $M_2=M_{q_0}$, and $f_i=f|_{M_i}$ for $i=1,2$.
\end{proof}


\noindent {\em Proof of Theorem \ref{thm:genstressPartial}}
Let $F=F^{G,\F}$, $\W=S_1\times S_2\times\ldots S_n$ and put
$f=F|_{\W}$. By Proposition \ref{prop:bob3.3} there exist open
neighbourhoods $N_{p_0}$ of $p_0$ and $N_{p_0}$ of $q_0$ in $\W$ and
a diffeomorphism $g:N_{q_0}\to N_{p_0}$ such that $g(q_0)=p_0$ and,
for all $q\in N_{q_0}$, $f(g( q))=f(q)$. Taking differentials at
$q_0$ we obtain
$df_{q_0}(q)=df_{p_0}(dg_{q_0}(q))$ for all $q$ in the tangent space
$T\W_{q_0}$. Since the Jacobian matrix of $F$ evaluated at $p$ is $2R_{\F}(G,p)$
and $df_{p}(x)=dF_{p}(x)$ for all $p\in \W$ and all $x\in T\W_{p}$,
we can rewrite this equation
as $R_{\F}(G,q_0)\,q=R_{\F}(G,p_0)\,dg_{q_0}(q)$.  Thus
$(\omega,\lambda)\, R_{\F}(G,q_0)\,q=(\omega,\lambda)\,
R_{\F}(G,p_0)\,dg_{q_0}(q)$. Since $(\omega,\lambda)$ is an
equilibrium stress for $(G,p_0)$ we have $(\omega,\lambda)\,
R_{\F}(G,q_0)\,q=(\omega,\lambda)\, R_{\F}(G,p_0)\,dg_{q_0}(q)= 0\,
dg_{q_0}(q)=0$ for all $q\in T\W_{q_0}$. Hence $(\omega,\lambda)\,
R_{\F}(G,q_0)\in T\W_{q_0}^\perp$, the orthogonal complement of $T\W_{q_0}$ in $\mathbb{R}^{3n}$. Since a vector
$x\in \bR^{3n}$ belongs to $T\W_{q_0}^\perp$ if and only if $x=\delta
S(G,q_0)$ for some $\delta\in \bR^n$, we have $(\omega,\lambda)\,
R_{\F}(G,q_0)=(0,\delta)R_{\F}(G,q_0)$. Therefore $(\omega,\lambda')$ is an equilibrium stress of $(G,q_0)$ for $\lambda'=\lambda-\delta$.\qed

\end{document}